\pdfoutput=1    

\documentclass[11pt,a4paper]{article}

\usepackage[OT2,T1]{fontenc}
\usepackage[english]{babel}
\usepackage{amstext}
\usepackage{amsfonts}
\usepackage{amsthm}
\usepackage{amsmath}
\usepackage{graphicx}
\usepackage{geometry}
\usepackage{hyperref}
\hypersetup{colorlinks=true, linkcolor=red, citecolor=red, pdfstartview=FitH, linktocpage=false}

\newcommand{\beq}{\begin{equation*}}
\newcommand{\eeq}{\end{equation*}}
\newcommand{\beqn}{\begin{equation}}
\newcommand{\eeqn}{\end{equation}}

\theoremstyle{theorem}
\newtheorem*{satz}{Theorem}

\addto\captionsenglish{
  
}

\newcommand{\dd}{\mathrm{d}}
\newcommand{\Z}{\mathcal{Z}_n}
\newcommand{\R}{\mathcal{R}_{a,\,b}}

\newcommand{\FS}{\mathcal{F}}

\newcommand{\XVn}{X_{n}}
\newcommand{\E}{\mathrm{E}}
\newcommand{\V}{\mathrm{Var}}
\newcommand{\p}{p_{n}}

\begin{document}

\title{Geometric probabilities for a cluster of needles\\ and a lattice of rectangles}
\author{Uwe B\"asel} 
\date{}  	
\maketitle
\thispagestyle{empty}
\begin{abstract}
\noindent A cluster of $n$ needles ($1\leq n<\infty$) is dropped at random onto a plane lattice of rectangles. Each needle is fixed at one end in the cluster centre and can rotate independently about this centre.\\[0.2cm]
The distribution of the relative number of needles intersecting the lattice is shown to converge uniformly to the limit distribution as $n\rightarrow\infty$.\\[0.3cm]
\textbf{AMS Classification:} 60D05, 52A22, 78M05  \\[0.3cm]
\textbf{AMS 2000 Subject Classifications:} Geometric probability, stochastic geometry, random sets and  random convex sets, method of moments
\end{abstract} 

\section{Introduction}
We consider a cluster $\Z$ of $n$ needles thrown at random onto a plane lattice $\R$ of rectangles (see Fig.\ \ref{Bild1}). The fundamental cell of $\R$ is a rectangle with side lengths $a$ and $b$. All needles of $\Z$ are connected in the centre of $\Z$. Each needle has equal length 1.

\begin{figure}[h]
  \vspace{0.4cm}
  \begin{center}
    \includegraphics[scale=0.75]{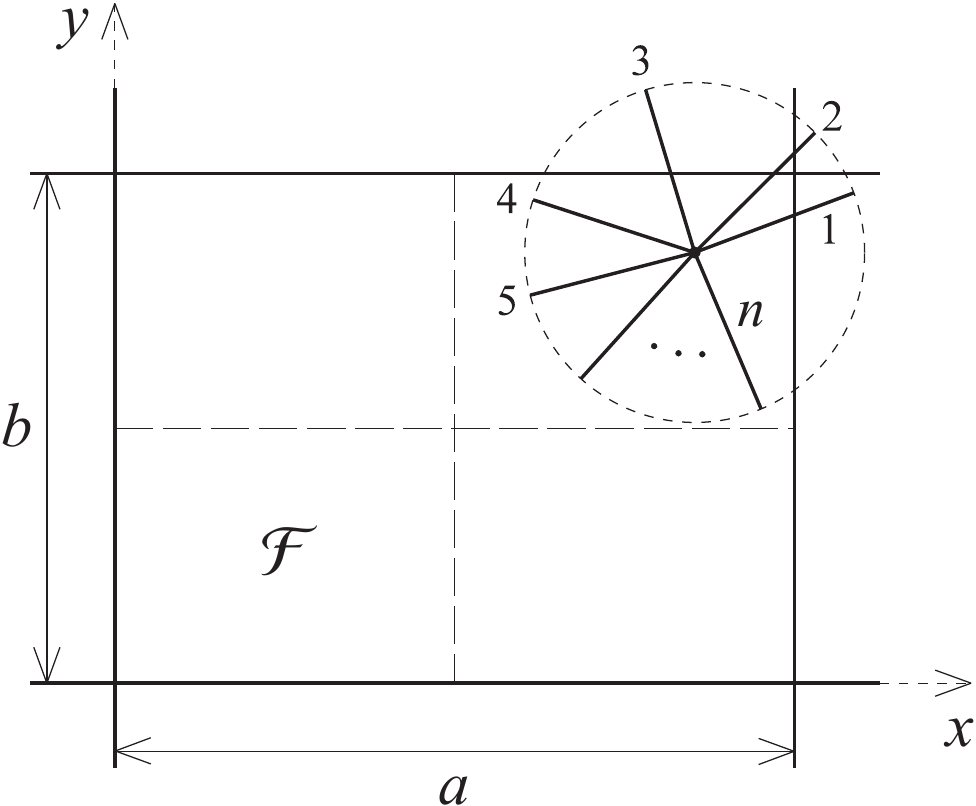}
  \end{center}
  \vspace{-0.3cm}
  \caption{\label{Bild1} Lattice $\R$ and cluster $\Z$}
\end{figure}
We assume $\min(a,b)\geq 2$ so that the cluster $\Z$ can intersect at most one of the vertical lines of $\R$ and (at the same time) one of the horizontal lines of $\R$ (except sets with measure zero). A \textit{random throw of $\Z$ onto $\R$} is defined as follows: After throwing $\Z$ onto $\R$ the coordinates $x$ and $y$ of the centre point are random variables uniformly distributed in $[0,a]$ and $[0,b]$ resp.; the angle $\phi_i$ between the $x$-axis and the needle $i$ is for $i\in\{1,\dots,n\}$ a random variable uniformly distributed in $[0,2\pi]$. All $n+2$ random variables are stochastically independent. In the following $\lambda:=1/a$ and $\mu:=1/b$ with $0\leq\lambda,\mu\leq1/2$ will be used.

\section{Intersection probabilities}
The intersection probabilities for this problem are derived in \cite{BaeDum3}. In this section the results are summarised, that are necessary for the following investigations.\\[0.2cm] 
$\p(i)$, $i\in\{0,\dots,2n\}$, denotes the probability of exactly $i$ intersections between $\Z$ and $\R$. Due to existing symmetries it is sufficient to consider only the subset $\mathcal{F}$ of the fundamental cell (Fig.\ \ref{Bild1}). For the calculations it is necessary to consider $\mathcal{F}$ as union of five subsets $\FS_1,\dots,\FS_5$ (see Fig.\ \ref{Bild2}):
\begin{eqnarray*}
  \FS_1 & = & \{(x,y)\in\mathbb{R}^2\,|\,1\leq x\leq a/2,\,1\leq y\leq b/2\}\,,\\[0.1cm]
  \FS_2 & = & \{(x,y)\in\mathbb{R}^2\,|\,0\leq x\leq 1,\,1\leq y\leq b/2\}\,,\\[0.1cm]
  \FS_3 & = & \{(x,y)\in\mathbb{R}^2\,|\,1\leq x\leq a/2,\,0\leq y\leq 1\}\,,\\[0.1cm]
  \FS_4 & = & \{(x,y)\in\mathbb{R}^2\,|\,0\leq x\leq 1,\,\sqrt{1-x^2}\leq y\leq 1\}\,,\\[0.1cm]
  \FS_5 & = & \{(x,y)\in\mathbb{R}^2\,|\,0\leq x\leq 1,\,0\leq y\leq \sqrt{1-x^2}\,\}\,.
\end{eqnarray*}

\begin{figure}[h]
  \vspace{0cm}
  \begin{center}
    \includegraphics[scale=0.6]{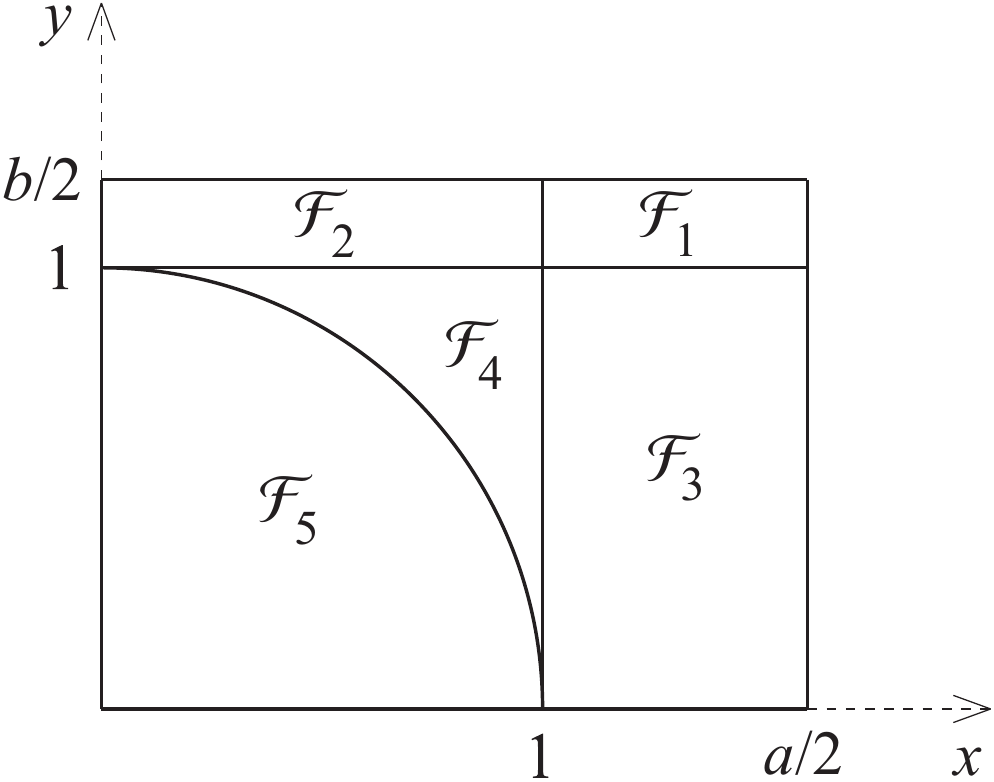}
  \end{center}
  \vspace{-0.3cm}
  \caption{\label{Bild2} $\mathcal{F}=\FS_1\cup\,\dots\,\cup\FS_5$}
\end{figure}

With $p_n(i\,|\,(x,y))$ we denote the conditional probability, that $\Z$ with centre point $(x,y)\in\FS$ has exactly $i$ intersections with $\R$. Considering the case distinctions for the subsets $\FS_m$ the probabilities are calculated with
\beq
  \p(i)
= \iint_{\FS}p_n(i\,|\,(x,y))f_1(x)f_2(y)\,\dd x\,\dd y
= \sum_{m=1}^5\iint_{\FS_m}p_n(i\,|\,(x,y))f_1(x)f_2(y)\,\dd x\,\dd y\,,
\eeq
where
\begin{equation*}
  f_1(x)=\left\{
  \begin{array}{ll}	
      2/a & \mbox{for}\;\;\, 0\leq x\leq a/2\;,\\[0.2cm]
	0 & \mbox{else}\;,  
  \end{array}\right.
  \qquad\mbox{and}\qquad
  f_2(y)=\left\{
  \begin{array}{ll}	
      2/b & \mbox{for}\;\;\, 0\leq y\leq b/2\;,\\[0.2cm]
	0 & \mbox{else}  
  \end{array}\right.
\end{equation*}
are the density functions of $x$ and $y$. We get
\begin{equation} \label{Formel_2}
  \p(i) = \frac{4}{a b}\sum_{m=1}^5\iint_{\FS_m}p_n(i\,|\,(x,y))\,\dd x\,\dd y
	  = 4\lambda\mu\,\sum_{m=1}^5\iint_{\FS_m}p_n(i\,|\,(x,y))\,\dd x\,\dd y\,.  
\end{equation} 
The conditional intersection probabilities for centre point $(x,y)\in\FS_1$ are given by
\begin{equation*}
  p_n(0\,|\,(x,y))=1 \;,\quad p_n(1\,|\,(x,y))=0 \;,\quad p_n(2\,|\,(x,y))=0\,.
\end{equation*}
For $(x,y)\in\FS_m\,,\;m\in\{2,3,4\}$, we have
\begin{align*}
  p_n(i\,|\,(x,y))
= {} & {n\choose i}\,q_1(x,y)^i\,(1-q_1(x,y))^{n-i}\,,\quad i\in\{0,1,\dots,n\}\,,\\[0.1cm]
  p_n(i\,|\,(x,y))
= {} & 0\,,\quad i\in\{n+1,\dots,2n\}\,,                        
\end{align*}
with
\begin{equation*}
  q_1(x,y)\;=\;\left\{\begin{array}{l}
    \displaystyle{\frac{1}{\pi}\arccos x}\;,\quad\mbox{if $(x,y)\in\FS_2$}\,,\\[0.4cm]
    \displaystyle{\frac{1}{\pi}\arccos y}\;,\quad\mbox{if $(x,y)\in\FS_3$}\,,\\[0.4cm]
    \displaystyle{\frac{1}{\pi}(\arccos x+\arccos y)}\;,\quad\mbox{if $(x,y)\in\FS_4$}\,.	
  \end{array}\right.
\end{equation*}
For $(x,y)\in\FS_5$, we have
\beq
  p_n(i\,|\,(x,y))
= \sum_{j=0}^{\lfloor i/2 \rfloor}{n \choose i-j}{i-j \choose j}\,
  q_2^j\,q_1^{i-2j}\,(1-q_1-q_2)^{n-i+j}\,,\quad
  i\in\{0,1,\dots,2n\}\,,          
\eeq
where $q_1 = q_1(x,y) = 1/2$, $q_2 = q_2(x,y) = \frac{1}{2\pi}(\arccos y-\arcsin x)$, and $\lfloor i/2 \rfloor$ denotes the integer part of $i/2$. 

\section{Distribution functions}

In the following let $\XVn$ denote the ratio
\begin{equation*}
  \frac{\mbox{number of intersections between $\Z$ and $\R$}}{n}\,.
\end{equation*}

We shall investigate the asymptotic behaviour of the distribution functions
\beq
  F_n(x)\;=\;P(\XVn\leq x)\;=\;\left\{
  \begin{array}{lll}
	0 & \mbox{for} & -\infty<x<0\,,\\[0.05cm]
	\displaystyle{\sum_{i=0}^{\lfloor nx \rfloor}\p(i)} & \mbox{for} & 0\leq x<2\,,\\[0.5cm]
	1 & \mbox{for} & 2\leq x<\infty\,,
  \end{array} \right.
\eeq
as $n\rightarrow\infty$, where $\p(i)$ is defined by (\ref{Formel_2}).  

\begin{satz} \label{Satz_1}
As $n\rightarrow\infty$, the random variables $X_n$ converge weakly to the random variable $X$, whose distribution function is given by
\begin{equation*}
  F(x)=\left\{
  \begin{array}{lll}
	0 & \mbox{for} & -\infty<x<0\,,\\[0.2cm]
	1-2(\lambda+\mu)\cos\pi x+2(2\cos\pi x-\pi x\sin\pi x)\lambda\mu & \mbox{for} & 
			0\leq x<\frac{1}{2}\,,\\[0.2cm]
	1+2\pi(x-1)\lambda\mu\sin\pi x & \mbox{for} & \frac{1}{2}\leq x<1\,,\\[0.2cm]
	1 & \mbox{for} & 1\leq x<\infty\,.  
  \end{array} \right.
\end{equation*}
Moreover, it holds the uniform convergence $\lim_{n\rightarrow\infty}\sup_{x\in\mathbb{R}}|F_n(x)-F(x)|=0$.
\end{satz} 

\begin{proof}
The proof of the weak convergence is based on the method of moments. According to the Fr\'echet-Shohat theorem (see e.g. \cite[pp. 81/82]{Galambos}) we have to show that for each $k\in\mathbb{N}$ the sequence of moments $\E(\XVn^k)=\int_{-\infty}^\infty x^k\,\dd F_n(x)$ converges to $\E(X^k)=\int_{-\infty}^\infty x^k\,\dd F(x)$ as $n\rightarrow\infty$ and the moments $\E(X^k)$, $k\in\mathbb{N}$, uniquely determine $F$.\\[0.2cm]
Since $F$ is a distribution function that is constant outside the interval $[0,1]$, it is uniquely determined by its moments. These moments are given by
\begin{align} \label{Momente_GV}
  \E(X^k) 
= {} & [2\pi(\lambda+\mu)-6\pi\lambda\mu]\int_{0}^{1/2}x^k\sin\pi x\;\dd x
		-2\pi^2\lambda\mu\int_0^{1/2}x^{k+1}\cos\pi x\;\dd x\nonumber\\
& +2\pi\lambda\mu\int_{1/2}^1 x^k\,[\sin\pi x-\pi(1-x)\cos\pi x]\,\dd x\,,
		\;k\in\mathbb{N}\;.
\end{align}
For the moments $\E(\XVn^k)$, $k\in\mathbb{N}$, we find
\begin{align*}
  \E(X_n^k) 
= {} & \int_{-\infty}^\infty x^k\:\dd F_n(x)
= \sum_{i=0}^{2n}\left(\frac{i}{n}\right)^k \p(i)\\
= {} & \sum_{i=0}^{2n}\left(\frac{i}{n}\right)^k 4\lambda\mu\,\sum_{m=1}^5
			\iint_{\FS_m}p_{n}(i\,|\,(x,y))\,\dd x\,\dd y\\
= {} & 4\lambda\mu\sum_{m=1}^5\iint_{\FS_m}\sum_{i=0}^{2n}\left(\frac{i}{n}\right)^k
			p_{n}(i\,|\,(x,y))\,\dd x\,\dd y\\
= {} & 4\lambda\mu\sum_{m=1}^5\iint_{\FS_m}\E(X_n^k\,|\,(x,y))\,\dd x\,\dd y\,,
\end{align*}
where $\E(X_n^k\,|\,(x,y))$ is the conditional $k$-th moment of $X_n$ given the cluster centre in $(x,y)$.\\[0.2cm]
Now let us consider the subsets $\FS_1,\ldots,\FS_5$:\\[0.2cm]
For centre point $(x,y)\in\FS_1$ and any $k\in\mathbb{N}$ we have $\E(X^k\,|\,(x,y))=0$ and therefore
\begin{equation*}
  \lim_{n\rightarrow\infty}\iint_{\FS_1}\E(X_n^k\,|\,(x,y))\,\dd x\,\dd y=0\;.
\end{equation*}
For centre point $(x,y)\in\mathcal{F}_m$, $m\in\{2,3,4\}$, and $i\in\{n+1,\dots,2n\}$ all conditional probabilities $p_n(i\,|\,(x,y))=0$. Hence we have
\begin{equation*}
  \E(X_n^k\,|\,(x,y)) = \sum_{i=0}^n\left(\frac{i}{n}\right)^k\,p_n(i\,|\,(x,y))\,
	= \sum_{i=0}^n\left(\frac{i}{n}\right)^k{n\choose i}\,q_1(x,y)^i(1-q_1(x,y))^{n-i}\,.
\end{equation*}
$\E(X_n^k\,|\,(x,y))$ is the Bernstein polynomial of the function $x^k$. In the interval $0\leq q_1(x,y)\leq 1$ it converges uniformly to $q_1(x,y)^k$ as $n\rightarrow\infty$ (see e.g. \cite[p. 222]{Feller}). It follows that $\E(X_n^k\,|\,(x,y))$ converges uniformly to $q_1(x,y)^k$ in $\mathcal{F}_m$, $m\in\{2,3,4\}$, that is
\begin{equation*}
  \lim_{n\rightarrow\infty}\,\sup_{(x,\,y)\,\in\,\mathcal{F}_m}\left|\E(X_n^k\,|\,(x,y))-q_1(x,y)^k\right|=0\,,
	\quad k\in\mathbb{N}\,.
\end{equation*}
Owing to the uniform convergence we can exchange limit and integral and get
\begin{eqnarray*}
 \lim_{n\rightarrow\infty}\iint_{\FS_m}\E(X_n^k\,|\,(x,y))\,\dd x\,\dd y  
   & = & \iint_{\FS_m}\lim_{n\rightarrow\infty}\E(X_n^k\,|\,(x,y))\,\dd x\,\dd y\\[0.2cm]
   & = & \iint_{\FS_m}q_1(x,y)^k\,\dd x\,\dd y\,,\quad m\in\{2,3,4\}\,.
\end{eqnarray*}
For $(x,y)\in\FS_2$ we have $q_1(x,y)=\frac{1}{\pi}\arccos x$, hence
\begin{equation*}
 \begin{array}{l}
  \displaystyle{\lim_{n\rightarrow\infty}\iint_{\FS_2}\E(X_n^k\,|\,(x,y))\,\dd x\,\dd y  
   \;=\; \int_{y=1}^{b/2}\int_{x=0}^1\left(\frac{\arccos x}{\pi}\right)^k\dd x\,\dd y}\\[0.5cm]
   \quad \;=\; \displaystyle{(b/2-1)\pi\int_0^{1/2}u^k\sin\pi u\;\dd u
   \;=\; \frac{1-2\mu}{2\mu}\,\pi\int_0^{1/2}u^k\sin\pi u\;\dd u\,.}
 \end{array} 		
\end{equation*}
For $(x,y)\in\FS_3$ we have $q_1(x,y)=\frac{1}{\pi}\arccos y$, hence
\begin{equation*}
 \begin{array}{l}
  \displaystyle{\lim_{n\rightarrow\infty}\iint_{\FS_3}\E(X_n^k\,|\,(x,y))\,\dd x\,\dd y  
   \;=\; \int_{x=1}^{a/2}\int_{y=0}^1\left(\frac{\arccos y}{\pi}\right)^k\dd y\,\dd x}\\[0.5cm]
   \quad \;=\; \displaystyle{(a/2-1)\pi\int_0^{1/2}u^k\sin\pi u\;\dd u
   \;=\; \frac{1-2\lambda}{2\lambda}\,\pi\int_0^{1/2}u^k\sin\pi u\;\dd u\,.}
 \end{array} 		
\end{equation*}
For $(x,y)\in\FS_4$ we have $q_1(x,y)=\frac{1}{\pi}(\arccos x+\arccos y)$ and therefore
\beq
  \lim_{n\rightarrow\infty}\iint_{\FS_4}\E(X_n^k\,|\,(x,y))\,\dd x\,\dd y
= \int_{y=0}^1\,\int_{x=\sqrt{1-y^2}}^1\left(\frac{\arccos x+\arccos y}{\pi}
		\right)^k\,\dd x\,\dd y\,.
\eeq
For centre point $(x,y)\in\FS_5$ we have
\begin{equation} \label{summe}
  \E(X_n^k\,|\,(x,y)) = \sum_{i=0}^{2n}\left(\frac{i}{n}\right)^k\;p_n(i\,|\,(x,y)) 
\end{equation}
with
\begin{equation*}
  p_n(i\,|\,(x,y)) = \sum_{j=0}^{\lfloor i/2 \rfloor}{n \choose i-j}{i-j \choose j}\,
	q_2^j\,q_1^{i-2j}\,(1-q_1-q_2)^{n-i+j}\,, 
\end{equation*}    
where $q_1=q_1(x,y)=1/2$ and $q_2=q_2(x,y)=\frac{1}{2\pi}(\arccos y-\arcsin x)$. Using the lemma in \cite[p.~219]{Feller} we show that $\E(X_n^k\,|\,(x,y))\rightarrow(q_1(x,y)+2q_2(x,y))^k$ uniformly as $n\rightarrow\infty$. At first we may write the expectation (\ref{summe}) as
\begin{equation} \label{integral}
  \E(X_n^k\,|\,(x,y))=\int_{t=0}^2\,t^k\,\dd F_n(t\,|\,(x,y))\,,
\end{equation}
where $F_n(t\,|\,(x,y))$ is the conditional distribution of the random variable $X_n$ for fixed cluster centre $(x,y)$.\\[0.2cm]
By $Z_i$, $i\in\{1,\ldots,n\}$, we denote the random number of intersections between needle $i$ and $\R$ given the cluster centre in $(x,y)$ and by $M_n$ the arithmetic mean $(Z_1+\ldots+Z_n)/n$. We have $\E(Z_i)=q_1+2q_2$, $\E(Z_i^2)=q_1+4q_2$ and therefore $\V(Z_i)=\E(Z_i^2)-[\E(Z_i)]^2=q_1+4q_2-(q_1+2q_2)^2$. Furthermore we find
\begin{equation*}
  \E(M_n) = \E(Z_1/n)+\ldots+\E(Z_n/n) = \E(Z_1) = q_1+2q_2\,.
\end{equation*}
Since the random variables $Z_1,\,\ldots\,,Z_n$ are independent and identically distributed we have
\begin{equation*}
  \V(M_n) = \V(Z_1/n)+\ldots+\V(Z_n/n) = \frac{1}{n}\,\V(Z_1) \;=\; \frac{q_1+4q_2-(q_1+2q_2)^2}{n}\,.
\end{equation*}  
We put $\mathcal{D}:=\{(q_1,q_2)\in\mathbb{R}\,|\,0\leq q_1\leq 1,\,0\leq q_2\leq 1-q_1\}$. The function $g: \mathcal{D}\rightarrow\mathbb{R}\,$, $g(q_1,q_2):=q_1+4q_2-(q_1+2q_2)^2$ has its maximum in the point $(1/2,\,0)$ with $g(1/2,\,0)=1$. Hence $\V(M_n)\leq 1/n$ and therefore $\V(M_n)\rightarrow 0$ as $n\rightarrow\infty$. From \cite[p. 219]{Feller} it follows that (\ref{integral}) converges uniformly to $(q_1+2q_2)^k$ as $n\rightarrow\infty$.\\[0.2cm]
Now we get
\begin{equation*}
 \begin{array}{l}
  \displaystyle{\lim_{n\rightarrow\infty}\iint_{\FS_5}\E(X_n^k\,|\,(x,y))\,\dd x\,\dd y
   \;=\; \iint_{\FS_5}\lim_{n\rightarrow\infty}\E(X_n^k\,|\,(x,y))\,\dd x\,\dd y}\\[0.5cm]
   \;=\; \displaystyle{\iint_{\FS_5}[q_1(x,y)+2\,q_2(x,y)]^k\,\dd x\,\dd y
     = \int_{y=0}^1\int_{x=0}^{\sqrt{1-y^2}}\left(\frac{\arccos x+\arccos y}{\pi}
		\right)^k\,\dd x\,\dd y\,.} 
 \end{array}
\end{equation*}
The sum of the integrals for $\FS_4$ and $\FS_5$ is given by
\beq
  \lim_{n\rightarrow\infty}\iint_{\FS_4\,\cup\,\FS_5}\E(X_n^k\,|\,(x,y))\,\dd x\,\dd y
= \int_{y=0}^1\int_{x=0}^1\left(\frac{\arccos x+\arccos y}{\pi}
		\right)^k\,\dd x\,\dd y\,. 
\eeq
We simplify this integral, that we denote by $I_{45}$. With the substitutions $\arccos x=\pi u$ and $\arccos y=\pi v$ ($\dd x=-\pi\sin\pi u\;\dd u$ and $\dd y=-\pi\sin\pi v\;\dd v$) it follows, that
\begin{equation*}
  I_{45} = \int_0^{1/2}\int_0^{1/2}(u+v)^k\sin\pi u\sin\pi v\;\dd u\,\dd v\,.
\end{equation*}
With $z:=u+v$ and considering $z$ as a constant we get $\dd z=\dd u$ and 
\begin{equation*}
  I_{45} = \int_{v=0}^{1/2}\int_{z=v}^{v+1/2}z^k\;\sin\pi(z-v)\;\sin\pi v\;\dd z\,\dd v\,.
\end{equation*}
Changing the order of integrations gives
\beq
  I_{45} 
= \int_{z=0}^{1/2}z^k\int_{v=0}^z\sin\pi(z-v)\;\sin\pi v\;\dd z\,\dd v\\
  + \int_{z=1/2}^1 z^k\int_{v=z-1/2}^{1/2}\sin\pi(z-v)\;\sin\pi v\;
		\dd z\,\dd v\,.
\eeq
The calculation of the inner integrals yields
\begin{eqnarray*}
  I_{45} 
   & = & \frac{\pi}{2}\int_0^{1/2}z^k\,[\sin\pi z-\pi z\cos\pi z]\,\dd z
		+ \frac{\pi}{2}\int_{1/2}^1 z^k\,[\sin\pi z-\pi(1-z)\cos\pi z]\,\dd z\,.
\end{eqnarray*}
As summary of the preceding results we get
\begin{align} \label{result}
  \lim_{n\rightarrow\infty}\E(\XVn^k)
= {} & 4\lambda\mu\left(\frac{1-2\mu}{2\mu}\,\pi\int_0^{1/2}x^k\sin\pi x\;\dd x
		+\frac{1-2\lambda}{2\lambda}\,\pi\int_0^{1/2}x^k\sin\pi x\;\dd x\right. \nonumber \\
& + \frac{\pi}{2}\int_0^{1/2}x^k\,[\sin\pi x-\pi x\cos\pi x]\,\dd x \nonumber \\
& + \left.\frac{\pi}{2}\int_{1/2}^1 x^k\,[\sin\pi x-\pi(1-x)\cos\pi x]
		\,\dd x\right) \nonumber \\[0.2cm]
= {} & [2\pi(\lambda+\mu)-6\pi\lambda\mu]\int_0^{1/2}x^k\sin\pi x\;\dd x
  - 2\pi^2\lambda\mu\int_0^{1/2}x^{k+1}\cos\pi x\;\dd x \nonumber \\
& + 2\pi\lambda\mu\int_{1/2}^1 x^k\,[\sin\pi x-\pi(1-x)\cos\pi x]\,\dd x\;.
\end{align}
The comparison of (\ref{result}) with (\ref{Momente_GV}) shows, that $\lim_{n\rightarrow\infty}\E(\XVn^k)=\E(X^k)$ for $k\in\mathbb{N}$. It follows that $F_n$ converges weakly to $F$ as $n\rightarrow\infty$.\\[0.2cm]
From the weak convergence it follows that $F_n$ converges uniformly to $F$ in all points of continuity of $F$. $F$ is a continuous function, if $\lambda=1/2$ and $\mu=1/2$. If $\lambda\not=1/2$ or $\mu\not=1/2$, $F$ is continuous except in the point 0. For this case we consider the convergence of $F_n(0)$ as $n\rightarrow\infty$.  The probability that $\Z$ does not intersect $\R$ is given by 
\begin{equation*}
  \p(0) = 4\lambda\mu\iint_{\FS}p_n(0\,|\,(x,y))\,\dd x\,\dd y
	  = 4\lambda\mu\iint_{\FS}q_0(x,y)^n\,\dd x\,\dd y\,,  
\end{equation*} 
where $q_0(x,y)$ denotes the probability that a single needle with one end point in the cluster centre $(x,y)$ has no intersections with $\R$. For almost every $(x,y)\in\FS\setminus\FS_1$ we have $q_0(x,y)<1$ and therefore $q_0(x,y)^n\rightarrow 0$ as $n\rightarrow\infty$. For every $(x,y)\in\FS_1$ we have $q_0(x,y)=1$. With Lebesgue's dominated convergence theorem we find
\begin{equation*}
  \lim_{n\rightarrow\infty}\p(0) 
	= 4\lambda\mu\lim_{n\rightarrow\infty}\iint_{\FS}q_0(x,y)^n\,\dd x\,\dd y
	= 4\lambda\mu\iint_{\FS_1}\dd x\,\dd y=(1-2\lambda)(1-2\mu)\,.
\end{equation*}
It follows that $F_n(0)\rightarrow F(0)$ as $n\rightarrow\infty$. Hence the convergence $F_n\rightarrow F$ is completely uniform.	 
So the proof is complete.
\end{proof}


\bigskip
\textbf{\large{Acknowledgment}}\\[0.2cm]
The author is grateful to Lothar Heinrich (University of Augsburg) for the fruitful discussion. Especially the proposed use of the Fr\'echet-Shohat theorem led to a simplification of the proof.

\bigskip

\begin{center}
\begin{tabular}{ccc}Uwe B\"ASEL\\[0.15cm]
Leipzig University of Applied Sciences\\
Department of Mechanical\\
and Energy Engineering,\\
04277 Leipzig, Germany &\\[0.15cm]
{\small uwe.baesel@htwk-leipzig.de} 
\end{tabular}
\end{center}

\end{document}